\documentclass[12pt]{amsart}
\usepackage{amsmath,amsthm,amssymb,mathrsfs}
\usepackage{graphicx}
\usepackage{mathbbol}
\usepackage{txfonts}
\usepackage[usenames]{color}
\usepackage{enumerate}
\usepackage{hyperref}
\usepackage[a4paper,margin=3cm]{geometry}

\newtheorem{theorem}{Theorem}[section]

\newtheorem{lemma}[theorem]{Lemma}
\newtheorem{proposition}[theorem]{Proposition}

\newtheorem{corollary}[theorem]{Corollary}

\newcommand{\be}{\begin{equation}}

 \newcommand{\ee}{\end{equation}}

\newcommand{\R}{\mathbb{R}}

\newcommand{\Leb}{{\mathscr L}}

\newcommand{\bs}{{\rm bs}}
\newcommand{\Prob}{\mathcal{P}}

\DeclareMathOperator{\Cb}{C_b}

\DeclareMathOperator{\Cbs}{C_\bs}
\DeclareMathOperator{\Lipb}{Lip_b}
\newcommand{\Ch}{{\sf Ch}}

\DeclareMathOperator{\dive}{div}

\newcommand{\meas}{\mathfrak{m}}

\DeclareMathOperator{\RCD}{RCD}
\DeclareMathOperator{\CD}{CD}

\newfont{\tmpf}{cmsy10 scaled 2500}

\begin{document}

\title[Continuity of Laplacian]{Continuity of nonlinear eigenvalues in $\CD(K,\infty)$ spaces with\\
respect to measured Gromov-Hausdorff convergence}

\author{Luigi Ambrosio
\and
Shouhei Honda
\and
Jacobus W. Portegies
} 




\thanks{Scuola Normale Superiore, \url{luigi.ambrosio@sns.it}}
\thanks{Tohoku University, \url{shonda@m.tohoku.ac.jp}}
\thanks{Eindhoven University, \url{j.w.portegies@tue.nl}}
\maketitle

\begin{abstract}
In this note we prove in the nonlinear setting of $\CD(K,\infty)$ spaces the stability of the Krasnoselskii spectrum of the Laplace operator 
$-\Delta$ under measured Gromov-Hausdorff convergence, under an additional compactness assumption satisfied, for instance, by
sequences of $\CD^*(K,N)$ metric measure spaces with uniformly bounded diameter. Additionally, we show that every element $\lambda$ in the Krasnoselskii spectrum is indeed an eigenvalue, namely there exists a nontrivial $u$ satisfying the eigenvalue equation $- \Delta u = \lambda u$.
\end{abstract}

\tableofcontents

\section{Introduction}

Given a smooth, closed, $n$-dimensional Riemannian manifold $M$, the Laplace-Beltrami 
operator is a self-adjoint operator on $L^2(M)$. Its associated quadratic form is the Dirichlet energy 
\[
\Ch(u) = \int_M |\nabla u|^2 d \mathcal{H}^n.
\]
Its spectrum is discrete and only consists of eigenvalues, that is values $\lambda \geq 0$ such that the equation 
\[
- \Delta u = \lambda u
\]
has a solution. These eigenvalues can be ordered, $0 = \lambda_1(M) < \lambda_2(M) \leq \lambda_3(M) \leq \dots$, (eigenvalues are repeated according to their multiplicity) and they can be found by the min-max formula
\begin{equation}
\label{eq:min-max-linear}
\lambda_k(M) := \inf_{\dim L \geq k} \sup_{v \in S(L)} \Ch(v),
\end{equation}
where the infimum is over subspaces $L$ of $L^2(M)$ of dimension at least $k$, and $S(L)$ denotes the unit sphere in $L$.

An easier version of the question that we consider in this note is the following: If a sequence of closed Riemannian manifolds $M_i$ converges to a limit space $X$, do the eigenvalues $\lambda_k(M_i)$ converge to a $\lambda_k(X)$
(suitably defined, if $X$ is not smooth)? Under Gromov-Hausdorff convergence, simple examples show that this is not true in general. In fact, this led Fukaya to the introduction of 
the concept of metric measure convergence, or measured Gromov-Hausdorff convergence \cite{Fukaya-Collapsing-1987}. 
Given uniform two-sided bounds on the sectional curvature, and a uniform upper bound on the diameter, the spectrum of the Laplace operator is indeed continuous under measured Gromov-Hausdorff convergence. 
Later, Cheeger and Colding \cite{Cheeger-Structure-III-2000} extended this continuity result to a setting of Riemannian manifolds with a uniform Ricci curvature lower bound and a diameter upper bound.
Without a curvature bound, the eigenvalues can only be guaranteed to be upper semi-continuous with respect to measured Gromov-Hausdorff convergence \cite{Fukaya-Collapsing-1987}.

In general, the limit spaces $X$ are not Riemannian manifolds. An important part of the results by Fukaya and Cheeger and Colding is that one can define a (nontrivial) Laplace operator on such limit spaces.
The Laplace operator is introduced through its quadratic form, by now usually referred to as the Cheeger energy.

The limit spaces fail to have regularity properties that allow for an analytic definition of the Ricci curvature. Yet, they inherit many properties (such as a Bishop-Gromov volume comparison theorem) from their approximating Riemannian manifolds, that are consequences of the assumed Ricci curvature lower bound. Soon after the results by Cheeger and Colding on the structure of Ricci limit spaces, Lott and Villani, and Sturm introduced a synthetic notion of a lower Ricci curvature bound \cite{Lott-Ricci-2009,Sturm-Geometry-I-2006,Sturm-Geometry-II-2006}, the so-called $\CD(K,\infty)$ condition, that can be stated for general metric measure spaces. This condition implies many of the properties that in the smooth case are a consequence of the Ricci curvature lower bound.

The first author, Gigli and Savar\'{e} showed how the Cheeger energy also induces a non-trivial Laplace operator on $\CD(K,\infty)$ spaces \cite{Ambrosio-Gigli-Savare11}. In general, however, the Laplace operator and the associated heat flow are non-linear. 
Indeed, a smooth compact Finsler manifold $(M,F)$ is an example of a $\CD(K,\infty)$ space, for some $K\in  \mathbb{R}$, and in this case the Cheeger energy agrees with the energy introduced by Shen in \cite{Shen-Non-Linear-1998}
\[
\Ch(u) = \int_M F^*( du )^2 d \mathcal{H}^n,
\]
where $F^*$ is the dual Finsler norm.
In this case the Laplace operator $-\Delta$ is defined as the $L^2$-gradient of the halved energy, that is 
\[
\int_M (-\Delta u) v d\meas = D\frac 12 \Ch(u)(v)
\]
for all $v \in L^2(\meas)$. This Laplace operator is linear if and only if the Finsler manifold is in fact Riemannian.

To rule out Finsler geometries, Gigli, Savar\'{e} and the first author defined $\RCD(K,\infty)$ spaces as $\CD(K,\infty)$ spaces for which the heat flow (and equivalently the Laplace operator) are linear \cite{Ambrosio-Gigli-Savare14}. 

A next natural question was whether the same continuity properties for the spectrum also hold for $\RCD(K,\infty)$ spaces. 
In \cite{Gigli-Convergence-2015}, Gigli, Mondino and Savar\'{e} proved the spectral stability under measured Gromov-Hausdorff convergence for $\RCD(K,\infty)$ spaces. 
The main ingredient is the Mosco convergence of the Cheeger energies on the approximating spaces to the Cheeger energy on the limit space.
Recently, the first and second author proved the Mosco convergence of the $p$-Cheeger energy and the continuity of the first eigenvalue of the $p$-Laplacian on $\RCD(K,\infty)$ spaces \cite{Ambrosio-Honda16}.

The purpose of this note is to extend the continuity result of Gigli, Mondino and Savar\'{e} to the setting of $\CD(K,\infty)$ spaces. 
However, since just as in the case for Finsler manifolds the Laplace operator is in general nonlinear, we should specify what we mean by eigenvalues.

We will say that $u$ is an eigenfunction and $\lambda$ is an eigenvalue if  they satisfy the eigenvalue equation
\[
- \Delta u  = \lambda u.
\]
We will recall the precise meaning of the Laplace operator in the next section.
We mentioned at the beginning of the introduction that for Riemannian manifolds, all eigenvalues $\lambda_k$ can be found through the min-max formula (\ref{eq:min-max-linear}). Even though for Finsler manifolds the numbers $\lambda_k$ are still invariants, they do not necessarily correspond to values $\lambda$ for which the eigenvalue equation $-\Delta u = \lambda u$ has a solution.

On the other hand, in the case of Finsler manifolds, eigenvalues still exactly correspond to critical values of the Cheeger energy restricted to the unit sphere. Moreover, because the Cheeger energy is even, the eigenvalues correspond to critical values of the (normalized) energy restricted to $\mathbb{RP}^{\infty} \subset H$. The topology of $\mathbb{RP}^{\infty}$ can then then be leveraged in a Morse-theoretic or mountain-pass approach to finding critical points.

In particular, from now on we define for $k \in \mathbb{N}$
\begin{equation}
\label{eq:min-max-nonlinear}
\lambda_k := \inf_{\gamma(V) \geq k} \sup_{v \in V} \Ch(v)
\end{equation}
where now the infimum is over the collection of compact, symmetric subsets $V$ of the unit sphere in $L^2(\meas)$ that have Krasnoselskii genus at least $k$ \cite{Krasnoselskii-Topological-1964}, see also \cite{Rabinowitz-Aspects-1971,Struwe-Variational-2008}. 
One can think of the Krasnoselskii genus as the ``essential dimension" of a subset. 
Gromov discusses a slightly different way to assign an essential dimension \cite{Gromov-Dimension-1988}.
For the subsets considered, the Krasnoselskii genus corresponds to the Lusternik-Schnirelmann category of the sets in projective space.
Equivalently, the minimization in the min-max problem (\ref{eq:min-max-nonlinear}) is over subsets in projective space that have Lusternik-Schnirelmann category at least $k$. 

This Morse-theoretic approach to finding critical points has many applications, such as the classical result by Lusternik and Schnirelmann on the existence of three distinct simple closed geodesics on a Riemannian manifold with the topology of a sphere \cite{Lusternik-Probleme-1929,Ballmann-Satz-1978},  the existence of minimal surfaces in Riemannian manifolds \cite{Pitts-Existence-1981} and the existence of infinitely many solutions to semi-linear elliptic equations, see for instance \cite{Hempel-Multiple-1970}. 
For a list of important references we refer to Struwe \cite{Struwe-Variational-2008}.
 
Since it is so close to the topic of this paper, let us single out the following result. 
In the context of the $p$-Laplacian on a bounded domain $\Omega$ in $\mathbb{R}^n$, Szulkin \cite{Szulkin-Ljusternik-1988} proved the existence of infinitely many pairs $(\lambda, u)$, satisfying the system
\[
\begin{cases}
- \dive ( |\nabla u |^{p-2} u ) = \lambda |u|^{p-2} u & \text{in } \Omega\\
u = 0 & \text{on } \partial \Omega\\
\frac{1}{p} \int |\nabla u|^p =1.
\end{cases}
\]

In fact, Szulkin's result is an application of a more general framework that he sets up. One might wonder if this framework, or the framework described by Struwe \cite{Struwe-Variational-2008} is general enough to encompass our setting, and therefore implies immediately that the $\lambda_k$'s defined through the min-max problem (\ref{eq:min-max-nonlinear}) are eigenvalues.

In the context of smooth Finsler manifolds, it was noted by Shen that this general framework can indeed be applied, and that the $\lambda_k$ as defined by the min-max problem (\ref{eq:min-max-nonlinear}) are eigenvalues of the Laplace operator \cite{Shen-Non-Linear-1998}.  However, in general, there may be eigenvalues $\lambda$ for which there is no $k$ such that $\lambda = \lambda_k$.
This framework requires the ``energy" to be of class $C^1$. The Cheeger energy generally does not have this regularity in $\CD(K,\infty)$ spaces. For this reason, we need to modify the standard arguments. In particular, rather than constructing a pseudo-gradient flow, we directly work with the gradient flow that is provided by the general Brezis-Komura theory of gradient flows on Hilbert spaces.

The two main results of this note are the following.
\begin{itemize}
\item We show in Theorem \ref{th:EigenvalueExistence} that when the sublevel sets of the Cheeger energy $\Ch$ are compact, the values $\lambda_k$ correspond to eigenvalues of $\Ch$. That is, there exists a nontrivial function $u_k \in L^2(\meas)$ such that
\[
- \Delta u_k = \lambda_k u_k.
\]
In fact, we also include a statement about the multiplicity of such eigenvalues.
\item We show in Theorem \ref{tmain} that when a sequence of $\CD(K,\infty)$ spaces converges in the measured Gromov-Hausdorff sense to a limit $\CD(K,\infty)$ space, the Krasnoselskii eigenvalues $\lambda_k$ 
converge to those on the limit space. 
\end{itemize}

The additional difficulty in this nonlinear context for proving the stability of the eigenvalues comes from the fact that Mosco convergence alone is not sufficient to prove this stability.
In the linear context, it suffices to approximate a finite number of functions on the limit space, namely the eigenfunctions, in a way that is guaranteed by Mosco convergence. However, in the nonlinear context, a whole family of functions needs to approximated in a continuous fashion to get the necessary estimates.

Finally, we conclude by pointing out some potential extensions.
Since the setting of the paper is nonlinear, it would be interesting to investigate also the continuity of the spectrum in the case
of the $p$-Cheeger energies $\Ch_p$, even in the case of $\RCD(K,\infty)$ spaces. 
Additionally, it is still an open question whether the values in the Krasnoselskii spectrum for the $p$-Cheeger energy, $p \neq 2$, are eigenvalues.

It would also be interesting to extend our results from probability to $\sigma$-finite measures, since most
of the results we use (in particular those in \cite{Gigli-Convergence-2015}) are already available in this more general setting.

\smallskip\noindent
{\bf Acknowledgements.} The first author acknowledges the support of the MIUR PRIN 2015 grant. The second author
acknowledges the support of the 
JSPS Program for Advancing Strategic International Networks to Accelerate the Circulation of Talented Researchers, the Grantin-Aid for Young Scientists (B) 16K17585 and the warm hospitality of SNS. The third author thanks Mark Peletier, 
Georg Prokert and Oliver Tse for helpful discussions and the SNS for its hospitality.

\section{Notation and preliminary results}

Throughout this paper, a \emph{metric measure space} is a triple $(X,d,\meas)$, where $(X,d)$
is a complete and separable metric space and $\meas$ is a Borel probability measure in $X$ with ${\rm supp\,}\meas=X$. 
We denote by $\Cb(X)$ (resp. $\Cbs(X)$) the space of bounded continuous (resp. bounded continuous
with bounded support) functions in $X$. Analogously, we denote by $\Lipb(X)$ the space of bounded
Lipschitz functions on $X$.

In our setting, we are dealing with a sequence $(\meas_i)$  of probability measures
weakly convergent to a probability measure $\meas$ in a metric space $(Z,d)$, 
namely in duality with $\Cb(Z)$.

Assuming that $f_i$ in suitable Lebesgue spaces relative to $\meas_i$ are given, we recall the notions
of weak and strong convergence for $f_i$, see also \cite{Honda2},
\cite{Gigli-Convergence-2015} and \cite{Ambrosio-Stra-Trevisan} for many more properties of the
weak/strong convergence across variable measure spaces.

\smallskip\noindent
{\bf $L^p$-weak convergence.} Let $p\in (1,\infty)$. We say that $f_i\in L^p(\meas_i)$ 
$L^p$-weakly converge to $f \in L^p(\meas)$ if $f_i\meas_i$ weakly converge to $f \meas$ in 
duality with $\Cb(Z)$, with
\begin{equation}\label{eq:bound-lp-norms} 
\sup_i \left\| f_i \right\|_{L^p(\meas_i)} < \infty.
\end{equation}

It is not difficult to prove that any sequence $(f_i)$ satisfying \eqref{eq:bound-lp-norms} has a $L^p$-weakly convergent
subsequence.
 
\smallskip\noindent
{\bf $L^p$-strong convergence.} Let $p\in (1,\infty)$. We say that $f_i\in L^p(\meas_i)$  
$L^p$-strongly converge to $f\in L^p(\meas)$ if, in addition to weak $L^p$-convergence, one has 
$\limsup_i\|f_i\|_{L^p(\meas_i)}\le\|f\|_{L^p(\meas)}$. 

It is easy to check that if $f_i$ $L^p$-strongly converge to $f$ and $g_i$ $L^q$-weakly converge to $g$, with
$q=p/(p-1)$, then 
$$
\lim_{i\to\infty}\int f_ig_i d\meas_i=\int fg d\meas.
$$

\smallskip\noindent
{\bf Slopes, subdifferentials and gradient flows of $\lambda$-convex functionals in Hilbert spaces.} Let 
$(H,\langle\cdot,\cdot\rangle)$ be a Hilbert space with norm $|\cdot|$ 
(in our case it will always be a Lebesgue $L^2$ space) and let $\Phi:H\to (-\infty,\infty]$. In the case when $\Phi$ is convex, a relevant
concept is the subdifferential $\partial\Phi(u)$, a closed and convex set (possibly empty) defined at all points $u\in \{\Phi<\infty\}$ by
$$
\partial\Phi(u):=
\left\{\xi\in H:\ \Phi(v)\geq\Phi(u)+\langle \xi,v-u\rangle\,\,\,\forall v\in H\right\}.
$$
 
For $\lambda\in\R$, we say that $\Phi$ is 
$\lambda$-convex if $\Phi-\frac\lambda 2|\cdot|^2$ is convex in $H$. The descending slope of $\Phi$, defined by
\[
|\partial\Phi|(v) := \limsup_{w \to v} \frac{\left(\Phi(v) - \Phi(w) \right)^+}{|v-w|} 
\] 
admits, thanks to $\lambda$-convexity, the representation \cite[Thm.~2.4.9]{Ambrosio-Gigli-Savare-05}
\begin{equation}\label{eq:RepresentationLocalSlope}
|\partial\Phi|(v) = \sup_{w \neq v} \left( \frac{\Phi(v)-\Phi(w)}{|v-w|}+\frac\lambda 2 |v-w|\right)^+.
\end{equation}
It follows immediately from \cite[Cor.~2.4.10]{Ambrosio-Gigli-Savare-05} that the descending slope of
$\lambda$-convex functionals is lower semicontinuous.
Another equivalent representation is (with the convention $\min\emptyset=\infty$)
\begin{equation}\label{eq:RepresentationLocalSlop2}
|\partial\Phi|(v) = \min\left\{|\xi|:\ \xi\in \partial_F\Phi(v)\right\}
\end{equation}
where $\partial_F\Phi$ is the \textit{Fr\'echet} subdifferential of $\Phi$ at $u$:
\begin{equation}\label{def:frediff}
\partial_F\Phi(u)=\left\{\xi\in H:\ 
\liminf_{t\to 0^+}\frac{\Phi(u+tv)-\Phi(u)}{t}\geq\int \xi v d\meas\quad\forall v\in H\right\}. 
\end{equation}
Notice that for convex functions $\Phi$, monotonicity of difference quotients yields $\partial_F\Phi=\partial\Phi$; more generally,
for a $\lambda$-convex $\Phi$, one has
\begin{equation}\label{def:frediff2}
\xi\in\partial_F\Phi(u)\quad\Longleftrightarrow\quad
\Phi(v)\geq\Phi(u)+\langle \xi,v-u\rangle+\frac\lambda 2|v-u|^2\,\,\,\forall v\in H.
\end{equation} 

If $\Phi:H\to (-\infty,\infty]$ is $\lambda$-convex and lower semicontinuous, the Brezis-Komura theory provides
existence and many more properties of the gradient flow $u(t)$ of $\Phi$ starting from $u$, namely the locally 
absolutely continuous map $u(t):(0,\infty)\to H$ such that $u(t)\to u$ as $t\to 0$ and 
$$
-u'(t)\in\partial_F \Phi(u(t))\qquad\text{for $\Leb^1$-a.e. $t>0$.}
$$
Equivalently, for $\lambda$-convex $\Phi$'s, \eqref{def:frediff2} can be used to show that 
the gradient flow can be characterized in terms of the \textit{evolution variational inequality}
\begin{equation}\label{eq:EVI}
\frac{d}{dt}\frac 12 |u(t)-v|^2\leq \Phi(v)-\Phi(u(t))-\frac\lambda 2|v-u(t)|^2
\qquad\text{for $\Leb^1$-a.e. $t>0$.}
\end{equation}
A systematic account of the theory can be found in \cite{Brezis}, we also quote \cite{Ambrosio-Gigli-Savare-05} for extensions of the theory
to the metric setting, based either on \eqref{eq:EVI} or on the energy dissipation points of view. 
We record in the following theorem the main properties of gradient flows we need.

\begin{theorem}
Assume that $\Phi:H\to (-\infty,\infty]$ is $\lambda$-convex and lower semicontinuous.
Then for all $u\in \overline{\{\Phi<\infty\}}$ there exists a unique gradient flow starting from $u$. The
induced semigroup ${\sf S}_t$ satisfies the following properties:
\begin{itemize}
\item[(1)] (contractivity and monotonicity) For all $t\geq 0$ one has
\begin{equation}\label{eq:contractivity}
\|{\sf S}_t u-{\sf S}_t v\|\leq e^{-\lambda t}\|u-v\|\qquad u,\,v\in\overline{\{\Phi<\infty\}},
\end{equation}
and $t\mapsto\Phi(u(t))$, $t\mapsto e^{\lambda t}|\partial\Phi|(u(t))$ are nonincreasing in $[0,\infty)$.
\item[(2)] (energy regularization) For all $t> 0$ one has (with the convention $(e^{\lambda t}-1)/t=1$ if $\lambda=0$)
\begin{equation}\label{eq:regularizing1}
\Phi(u(t))\leq \inf_{v\in H}\left\{ \Phi(v)+\frac{|u-v|^2}{2(e^{\lambda t}-1)/\lambda}\right\}.
\end{equation}
\item[(3)] (slope regularization) For all $t> 0$ one has
\begin{equation}\label{eq:regularizing2}
e^{-2\lambda^-t}|\partial\Phi|^2(u(t))\leq\inf_{v\in H}\left\{|\partial\Phi|^2(v)+
\frac{|u-v|^2}{t^2}+\frac{\lambda^-}{t^2}\bigl(\int_0^t|v-u(s)|^2 ds+ t|v-u(t)|^2\bigr)
\right\} .
\end{equation}
\item[(4)] (minimal selection) For $\Leb^1$-a.e. $t>0$ one has that $-u'(t)$ is the element with minimal
norm in $\partial_F\Phi(u(t))$.
\item[(5)] (energy identity) If $\Phi(u)<\infty$, then $t\mapsto\Phi(u(t))$ is locally absolutely continuous in $[0,\infty)$, with
$$
-\frac{d}{dt}\Phi(u(t))=|u'(t)|^2=|\partial\Phi|^2(u(t))\qquad\text{for $\Leb^1$-a.e. $t>0$.}
$$
\end{itemize}
\end{theorem}
\begin{proof} For the reader's convenience we provide the proof of \eqref{eq:regularizing2}, adapting 
\cite[Thm.~4.3.2]{Ambrosio-Gigli-Savare-05}, where the statement is given only for $\lambda=0$ (while \eqref{eq:regularizing1} is
fully proved therein). 

By the monotonicity of $e^{\lambda t}|\partial\Phi|(u(t))$, integrating in time, we get
  \begin{eqnarray*}
    \frac{t^2e^{2\lambda t}}2|\partial\Phi|^2(u(t))&\le&
    \int_0^t s e^{2\lambda s}|\partial\Phi|^2(u(s))ds\le
    -e^{2\lambda^+t}\int_0^t s\big(\Phi(u(s)\big)'ds\\&\le&
    e^{2\lambda^+t}\biggl[\int_0^t\Phi(u(s)) ds-t\Phi(u(t))\biggr],
  \end{eqnarray*}
  so that
  $$
  \frac{t^2e^{-2\lambda^- t}}2|\partial\Phi|^2(u(t))\le\int_0^t\Phi(u(s)) ds-t\Phi(u(t)).
  $$
  Now we use the inequality
  $$
  \int_0^t\Phi(u(s))\,ds\leq t\Phi(v)-\frac 12|u(t)-v|^2+\frac{|u-v|^2}{2}+\frac{\lambda^-}2\int_0^t|v-u(s)|^2 ds
  $$
  that comes from integration of \eqref{eq:EVI}, as well as the inequality 
  $$
  \Phi(v)-\Phi(u(t))\leq |\partial\Phi|(v)|v-u(t)|+\frac{\lambda^-}{2}|v-u(t)|^2
  $$
  that comes from \eqref{eq:RepresentationLocalSlope}, to get
  $$
  \frac{t^2e^{-2\lambda^- t}}2|\partial\Phi|^2(u(t))\le t|\partial\Phi|(v)|v-u(t)|
  -\frac 12|u(t)-v|^2+\frac{|u-v|^2}{2}+\frac{\lambda^-}2\bigl(\int_0^t|v-u(s)|^2 ds+ t|v-u(t)|^2\bigr).
  $$
  Eventually with Young's inequality we conclude.
  \end{proof}

\smallskip\noindent
{\bf Cheeger energies and heat flow.}
We recall basic facts about Cheeger energies and heat flow in metric measure spaces
$(X,d,\meas)$, see \cite{Ambrosio-Gigli-Savare11} 
and \cite{Gigli-Convergence-2015} for a more systematic treatment of this topic. For
$p\in (1,\infty)$ the $p$-Cheeger energy
$\Ch_p:L^p(\meas)\to [0,\infty]$ is the convex and $L^p(\meas)$-lower semicontinuous functional defined as follows:
\begin{equation}\label{eq:defchp}
\Ch_p(f):=\inf\left\{\liminf_{n\to\infty}\int |\nabla f_n|^pd\meas:\ \text{$f_n\in\Lipb(X)$, $\|f_n-f\|_p\to 0$}\right\},
\end{equation}
where $|\nabla f|$ denotes the slope, also called local Lipschitz constant (notice that we drop the factor
$p^{-1}$ in front of the integral, used in other papers on this topic).
The case $p=2$ plays an important role in the axiomatization of the so-called $\RCD(K,\infty)$ spaces
\cite{Ambrosio-Gigli-Savare14} and in the construction of the differentiable structure, see \cite{Gigli1}.
For this reason we use the disinguished notation $\Ch=\Ch_2$ and denote by $D(\Ch)$ its finiteness domain.
  
Another object canonically associated to $\Ch$ and then to the metric measure structure is the heat flow $h_t$, defined as the
$L^2(\meas)$ gradient flow of $\frac 12\Ch$, according to the above mentioned Brezis-Komura theory of gradient flows of lower semicontinuous
and convex functionals in Hilbert spaces. This theory provides a continuous contraction
semigroup $h_t$ in $L^2(\meas)$ with the Markov property, characterized by
$$
\frac{d}{dt}h_t f=\Delta h_tf\quad\text{in $L^2(\meas)$, for a.e. $t>0$},\qquad\lim_{t\to 0^+}h_t f=f
$$
for all $f\in L^2(\meas)$, where $-\Delta g$ is the element with minimal $L^2(\meas)$ norm in $\partial\tfrac 12\Ch(g)$.
We shall also use that, because of the $2$-homogeneity of $\Ch$, one has
 (see \cite[Prop.~4.15]{Ambrosio-Gigli-Savare11} for a proof when $\xi=-\Delta f$, the same proof
works with any $\xi\in\partial\Ch(f)$)
\begin{equation}\label{eq:nicelaplacian}
\Ch(f)=\int \xi fd\meas
\qquad\forall \xi\in\partial\tfrac{1}{2}\Ch(f).
\end{equation}

We shall also extensively use the typical regularizing properties which follow by \eqref{eq:regularizing1} and
\eqref{eq:regularizing2} (with $\lambda=0$)
\begin{equation}\label{eq:Brezis1}
\Ch(h_t f)\leq\frac{\|f\|_{L^2(X,\meas)}^2}{t},
\end{equation}
\begin{equation}\label{eq:Brezis2}
\|\Delta h_t f\|_{L^2(\meas)}^2\leq\frac{\|f\|_{L^2(\meas)}^2}{t^2}
\end{equation}
as well as the monotonicity property $\Ch(h_sf)\leq\Ch(h_t f)\leq\Ch(f)$ for $0\leq t\leq s$.

\smallskip
\noindent
{\bf $\CD(K,\infty)$ spaces.} Denote by $\Prob(X)$ the class of Borel probability measures in $(X,d)$
and set
$$
\Prob_2(X):=\left\{\mu\in\Prob(X):\ \int d^2(\bar x,x)d\meas(x)<\infty\,\,
\text{for some, and thus all, $\bar x\in X$}\right\}.
$$
We say that a metric measure space $(X,d,\meas)$
is a $\CD(K,\infty)$ metric measure space, with $K\in\R$, if
the Relative Entropy Functional ${\rm Ent}(\mu):\Prob_2(X)\to\R\cup\{\infty\}$ given by
$$
{\rm Ent}(\mu):=
\begin{cases}
\int\rho\log\rho d\meas&\text{if $\mu=\rho\meas\ll\meas$;}
\\ 
\infty &\text{otherwise}
\end{cases}
$$
is $K$-convex along Wasserstein geodesics in $\Prob_2(X)$. This means that for all
$\mu,\,\nu\in\Prob_2(X)$ there exists a constant speed geodesic $\mu_t:[0,1]\to\Prob_2(X)$
relative to $W_2$ with $\mu_0=\mu$, $\mu_1=\nu$ and
$$
{\rm Ent}(\mu_t)\leq (1-t){\rm Ent}(\mu_0)+t{\rm Ent}(\mu_1)-\frac{K}{2}t(1-t)W_2^2(\mu_0,\mu_1)
\qquad\forall t\in [0,1].
$$
Also, in $\CD(K,\infty)$ spaces we shall use the implication
\begin{equation}\label{eq:rajala}
\Ch(u)=0\quad\Longrightarrow\quad \text{$u=c$ $\meas$-a.e. in $X$, for some $c\in\R$.}
\end{equation}
Let us provide a justification of \eqref{eq:rajala} when $K\geq 0$.
By the chain rule, it suffices to show this implication when $u\geq 0$ and $\int u^2 d\meas=1$. Then, the
identification
$$
|\partial{\rm Ent\,}|(u^2\meas)=2\sqrt{\Ch(u)}
$$
provided in \cite[Thm.~9.3(i)]{Ambrosio-Gigli-Savare11} gives that $|\partial{\rm Ent\,}|(u^2\meas)=0$. By convexity, 
this yields that $u^2\meas$ is a minimizer of ${\rm Ent}$, whence $u=1$ $\meas$-a.e. in $X$. In the general case
we need to invoke the local Poincar\'e inequality of \cite{Rajala}.

\smallskip
{\bf $\CD^*(K,N)$ spaces.} For $N\geq 1$ and $K\in\R$, we denote by $\CD^*(K,N)$
the class of  metric measure spaces satisfying the reduced curvature-dimension condition and 
introduced in \cite{Bacher-Sturm}. This class includes the $\CD(K,N)$ class considered in 
\cite{Lott-Ricci-2009,Sturm-Geometry-I-2006,Sturm-Geometry-II-2006} (see \cite[Prop.~2.5]{Bacher-Sturm}) 
and it is contained in $\CD(K,\infty)$ (see \cite[Lem~9.13]{Ambrosio-Mondino-Savare15}).

These spaces satisfy the Bishop-Gromov comparison inequality, therefore are doubling as metric measure spaces.
In particular, since in our setting $\meas$ is finite, these spaces are compact whenever their diameter is finite.

\section{The Krasnoselskii spectrum}

For a Banach space $W$, denote 
\[
\mathcal{V}(W) = \left\{ V \subset W \middle|\ V \text{ closed and symmetric } \right\}.
\]
The Krasnoselskii genus $\gamma_{W}:\mathcal{V}(W) \to \mathbb{N}\cup \{ \infty \}$ is defined as follows. 
Let $V \in \mathcal{V}(W)$ be non-empty. 
If there exist $m \in \mathbb{N}$ and an odd function $h \in C^0(V;\mathbb{R}^m\setminus\{0\})$, we set
\[
\gamma_{W}(V) =
\inf\left\{ m \in \mathbb{N} \, \middle| \ \exists h \in C^0(V ; \R^m \backslash \{0\}), h \text{ odd } \right\}.
\]
Otherwise, we set $\gamma_{W}(V) = \infty$. Further, we define $\gamma_W(\emptyset) = 0$.
It is not difficult to check
\cite[Prop.~5.2]{Struwe-Variational-2008} that
\begin{equation}\label{eq:compagenus}
\gamma_W(S(L))={\rm dim\,}L
\end{equation}
for a finite-dimensional subspace $L$ of $W$, where here and in the sequel 
$$
S(L):=\left\{v\in L:\ \|v\|=1\right\}
$$
is the unit sphere of $L$.

For $k\geq 1$ we define also
\[
\mathcal{F}_k(W) = \left\{ V \in \mathcal{V}(W) \middle|\ \gamma_{W}(V) \geq k, \,\,V \subset S(W) \,\,
\hbox{\rm compact}\right\}.
\]
We will adopt the following ``nonlinear'' definition of spectrum
\[
\lambda_k(\Ch) = 
\inf_{V \in \mathcal{F}_k(L^2(\meas)) } \sup_{u \in V}\Ch (u). 
\]
Notice that still $\lambda_1=0$, and that
there may be critical values of the energy $\Ch$ that do not correspond to a value of $\lambda_k$ for any $k \in \mathbb{N}$.

In the degenerate case when $X$ consists of a single point, $\meas$ is a Dirac mass, $\Ch$ is identically null and
$L^2(\meas)$ is $1$-dimensional, so the above definitions give $\lambda_1(\Ch)=0$, 
$\lambda_k(\Ch)=
+\infty$ for all $k>1$. 
Recall also that, for $\CD(K,\infty)$ metric measure
spaces $(X,d,\meas)$, either $X$ consists of a single point, or
$\meas$ has no atom.

For the convenience of the reader, we include Proposition 5.4 of \cite{Struwe-Variational-2008} on properties of the Krasnoselskii genus.

\begin{proposition}[{cf. \cite{Struwe-Variational-2008}}]
\label{pr:PropertiesGenus}
Let $V,\, V_1,\, V_2 \in \mathcal{V}(W)$ and let $h: W \to W$ be continuous and odd. Then the following properties hold:
\begin{enumerate}
	\item $\gamma(V) \geq 0$; $\gamma(V) = 0$ if and only if $V = \emptyset$.
	\item $V_1 \subset V_2 $ implies $ \gamma(V_1) \leq \gamma(V_2)$.
	\item $\gamma(V_1 \cup V_2) \leq \gamma(V_1) + \gamma(V_2)$.
	\item $\gamma(V) \leq \gamma\left(\overline{h(V)}\right)$. 
	\item If $V$ is compact and $0 \notin V$, then $\gamma(V) < \infty$ and there is a neighborhood $N$ of $V$ in $W$ such that $\overline{N} \in \mathcal{V}(W)$ and $\gamma(V) = \gamma(\overline{N})$.
\end{enumerate}
\end{proposition}

The last item implies the upper semicontinuity of the genus w.r.t. Hausdorff convergence.

\begin{proposition}
\label{pr:SemicontinuityGenus}
Let $W$ be a Banach space and suppose compact sets $F_i \in \mathcal{V}(W)$ converge to a compact set $F \in \mathcal{V}(W)$ in the Hausdorff distance, such that $0 \notin F$. Then
\[
\gamma(F) \geq \limsup_{i \to \infty} \gamma(F_i).
\]
\end{proposition}
\begin{proof}
According to the previous proposition, given a compact set $F \in \mathcal{V}(W)$ with $0\notin F$, then $\gamma(F)<\infty$ and there exists a neighborhood $N$ of $F$ such that $\overline{N} \in \mathcal{V}(W)$ and $\gamma(\overline{N}) = \gamma(F)$. 
For $i$ large enough, $F_i \subset N$, and therefore for $i$ large enough, $\gamma(F_i) \leq \gamma(\overline{N}) = \gamma(F)$.
\end{proof}

\section{Setup and main result}

We consider $\CD(K,\infty)$ metric measure spaces $(X_i,d_i,\meas_i)$ and $(X,d,\meas)$, that are isometrically embedded in a common metric space $(Z,\rho)$, and such that $\meas_i \to \meas$ weakly in duality with $\Cb(Z)$. We denote by $\Ch^i,\,\Ch$ the corresponding
Cheeger energies, by $\Delta_i,\,\Delta$ their laplacians, by $h^i_t$, $h_t$ the heat flows. 
As illustrated in \cite{Gigli-Convergence-2015}, this ``extrinsic'' notion of convergence is equivalent to many others,
and it reduces to measured Gromov-Hausdorff convergence in the class of uniformly doubling metric measure spaces.
In addition, the main result of \cite{Gigli-Convergence-2015} is the Mosco convergence of $\Ch^i$ to $\Ch$, namely:
\begin{itemize}
\item[(a)] for all $f\in L^2(\meas)$ there exist $f_i$ $L^2$-strongly convergent to $f$ with
$\limsup_i\Ch^i(f_i)\leq\Ch(f)$;
\item[(b)] $\liminf_i\Ch^i(f_i)\geq\Ch(f)$ whenever $f_i$ $L^2$-weakly converge to $f$.
\end{itemize}
Notice that the Mosco convergence differs from $\Gamma$-convergence because different notions of convergence
are considered in (a) and (b).

In the proof of the lower semicontinuity theorem we shall also need the following compactness
result w.r.t. $L^2$-strong convergence from
\cite[Thm.~6.3]{Gigli-Convergence-2015}:

\begin{theorem}[Compactness]\label{tcompa}
If $f_i\in D(\Ch^i)$ satisfy 
\begin{equation}\label{eq:april30}
\sup_i\|f_i\|^2_{L^2(\meas_i)}+\Ch^i(f_i)<\infty,\qquad
\lim_{R\to\infty}\sup_i\int_{Z\setminus B_R(\bar x)}f_i^2d\meas_i=0
\end{equation}
for some $\bar x\in X$, then $f_i$ admit a $L^2$-strongly convergent subsequence. The second condition in \eqref{eq:april30} is
implied by the first one  if the estimate
\begin{equation}\label{eq:wlsi}
\int \rho^2(x,\bar x) f^2(x)d\meas_i(x)\leq A\int f^2d\meas_i+B\Ch^i(f)\qquad\forall f\in L^2(\meas_i)
\end{equation}
holds with $A,\,B\geq 0$, $\bar x\in Z$ independent of $i$. Finally, \eqref{eq:wlsi} holds if
either $K>0$, or $(Z,\rho)$ has bounded support.
\end{theorem}

By applying the previous theorem to a constant sequence of spaces, the compactness of the sublevels 
is true whenever \eqref{eq:april30} holds; the latter
is true if either $K>0$, or $(X,d)$ has bounded support, or more generally an inequality of the
form \eqref{eq:wlsi} holds.
Thanks to the compactness of the sublevels of $\|\cdot\|_{L^2(\meas)}+\Ch$,
and using the finiteness of the genus of compact sets,
one can prove that the spectrum provided by the Krasnoselskii eigenvalues is discrete.

\begin{corollary}\label{cor:div}
Assume that for all $s,\,t\geq 0$ the sets
\begin{equation}\label{eq:mozzano6}
E^{s,t} := \left\{u \in L^2(\meas) :\ \|u\|_{L^2(\meas)} \le s,\,\,\Ch(u) \le t \right\} 
\end{equation}
are compact in $L^2(\meas)$. Then
$$
\lim_{k \to \infty}\lambda_k(\Ch)=+\infty.
$$
\end{corollary}
\begin{proof} By contradiction, assume that $\lambda_k(\Ch)< M<\infty$ for all $k\geq 1$. 
Let $V_k\in\mathcal{F}_k(L^2(\meas))$
be satisfying
$$
\sup_{u\in V_k}\Ch(u)\leq M.
$$ 
By Theorem~\ref{tcompa}, the closure $A$ of the union $\cup_k V_k$ is compact. From 
\cite[Prop.~5.4]{Struwe-Variational-2008} then we obtain that
$\gamma(A)<\infty$, and this contradicts the fact that $\gamma(A)\geq\gamma(V_k)\geq k$.
\end{proof}

Our main result is the following.

\begin{theorem}[Convergence of eigenvalues]\label{tmain}
Under the above assumptions on the $\CD(K,\infty)$ spaces, one has
\begin{equation}\label{eq:mozzano3}
\limsup_{i \to \infty} \lambda_k(\Ch^i ) \leq \lambda_k(\Ch)
\qquad\forall k\geq 1.
\end{equation}
In addition, if \eqref{eq:wlsi} holds with constants $A,\,B\geq 0$ and $\bar x\in Z$
independent of $i$, one has
\begin{equation}\label{eq:mozzano4}
\liminf_{i \to \infty} \lambda_k(\Ch^i ) \geq \lambda_k(\Ch)
\qquad\forall k\geq 1.
\end{equation}
\end{theorem}

A simple continuity argument then gives also uniform bounds on $\lambda_k$ in compact families
of metric measure spaces and provides a uniform rate of growth of $\lambda_k$. 

\begin{corollary}\label{cmain}
Let $K\in\R$ and $N\in [1,\infty)$.
For any $k\geq 2$ there exist positive and finite constants $C_1(K,N,k)$, $C_2(K,N,k)$ such that
\begin{equation}\label{eq:unif bound}
C_1(K, N,k) \le \lambda_k(\Ch) \le C_2(K, N,k),
\end{equation}
for any $\CD^*(K,N)$-space $(X,d,\meas)$ with $\mathrm{diam}\,X=1$. In addition,
$C_1(K,N,k)\to +\infty$ as $k\to +\infty$.
\end{corollary}
\begin{proof}
Let us denote by $\mathcal{M}(K, N)$ the set of all isometry classes of $CD^*(K, N)$-spaces $(X, d,\meas)$ with $\mathrm{diam}\,X=1$ and let us consider a function $\lambda_k: \mathcal{M}(K,N) \to [0,\infty)$.
The compactness of $\mathcal{M}(K, N)$ with respect to the mGH-convergence and the continuity
of $(X,d,\meas)\mapsto\lambda_k(\Ch)$ provided by 
Theorem~\ref{tmain} yield (\ref{eq:unif bound}), with $C_i$ equal to the minimal and maximal values.
The same compactness argument gives also the positivity of $C_1(K,N,k)$, namely the
positivity of $\lambda_2(\Ch)$ for any fixed $\CD^*(K,N)$ metric measure space. Indeed, if $\lambda_2(\Ch)=0$
by compactness and upper semicontinuity of the genus we can find a set $V$ with genus at least 2 such that 
$\Ch\equiv 0$ on $V$; on the other hand,
\eqref{eq:rajala} forces $V$ to consist of the functions $\{\pm 1\}$, and this set has genus equal to 1.

Finally, if $C_1(K, N, k)$ are bounded (with respect to $k$), then we can find a  subsequence $k(p)$ and
minimizers $(X_p,d_p,\meas_p) \in \mathcal{M}(K, N)$ of $\lambda_{k(p)}$ convergent to 
$(X,d,\meas) \in \mathcal{M}(K, N)$. Then, lower semicontinuity gives
that the spectrum of $(X,d,\meas)$ is bounded, contradicting Corollary~\ref{cor:div}.
\end{proof}

\section{Upper semicontinuity of the spectrum}

We shall use the following lemma, see \cite[Thm.~6.11]{Gigli-Convergence-2015} for $L^2$
convergence, see also \cite[Lem.~5.4 and Cor.~5.5]{Ambrosio-Honda16} for the case of $\RCD(K,\infty)$ spaces;
we show how the argument extends to $\CD(K,\infty)$ spaces.

\begin{lemma}\label{le:StrongLiftHeatKernel}
Let $v_i \to v$ strongly in $L^2$. Then, for every $t> 0$, $h_t^i v_i \to h_t	v$ strongly in $H^{1,2}$, i.e.
$h_t^iv_i$ $L^2$-strongly converge to $h_t v$ and $\limsup_i\Ch^i(h_t^i v_i)\leq\Ch(h_t v)$.
\end{lemma}
\begin{proof}
In order to prove strong $H^{1,2}$ convergence, taking the identities \eqref{eq:nicelaplacian} into account,
one has
$$
-\Ch^i(h^i_t v_i)=\int h^i_t v_i\Delta_i h^i_t v_i d\meas_i,\qquad
\Ch(h_t v)=\int h_t v\xi d\meas,
$$
for all $\xi\in \partial\tfrac{1}{2}\Ch(h_t v)$. Hence, it is sufficient to show that any 
$L^2$-weak limit point of $-\Delta_i h^i_t v_i$ $L^2$ belongs to $\partial\tfrac{1}{2}\Ch(h_t v)$ (notice that
the regularization estimate \eqref{eq:Brezis2} gives that $\|\Delta_i h^i_t v_i\|_{L^2(\meas_i)}$ are uniformly bounded). 
This follows at once from 
$$
\frac 12\Ch^i(w)\geq \frac 12\Ch^i(h_t^i v_i)-\int \Delta h^i_tv_i (w-h_t^i v_i)d\meas_i\qquad\forall w\in L^2(\meas_i)
$$
and from the Mosco convergence of $\Ch^i$ to $\Ch$.
\end{proof}

\begin{lemma}
	\label{le:ContructionIsometry}
If $\meas_i$ are not Dirac masses, there exist linear isometries
\[
\pi_i: L^2(\meas) \to L^2(\meas_i) 
\]
with the additional property that 
\[
\pi_i(v_i) \to v\quad\text{strongly in $L^2$, whenever $v_i \to v$ in $L^2(\meas)$.}
\]
The functions $h_t^i\circ\pi_i:L^2(\meas)\to L^2(\meas_i)$ are $1$-Lipschitz and
$1$-homogeneous. In addition, thanks to Lemma~\ref{le:StrongLiftHeatKernel}, 
for all $t>0$ one has the stronger property
\[
\text{$h_t^i \pi_i(v_i) \to h_t v$ strongly in $H^{1,2}$, whenever $v_i \to v$ in $L^2(\meas)$.}
\]
\end{lemma}
\begin{proof}
The construction of the linear isometries $\pi_i: L^2(\meas) \to L^2(\meas_i)$ goes as follows.
Since $\meas_i\to\meas$ weakly in $(Z,\rho)$, if we use the bounded cost function 
$\tilde\rho^2(x,y)=\min\{1,\rho^2(x,y)\}$ to define the Wasserstein distance we obtain
$W_2^2( \meas, \meas_i)\to 0$ as $i\to\infty$.
Select now $2^{-i}$-almost optimal transport maps $T_i: X_i \to X$ from $\meas_i$ to $\meas$, namely Borel 
maps satisfying $(T_i)_\#\meas_i=\meas$ and
\[
\int \tilde\rho( x, T_i(x) )^2 d \meas_i (x) 
\leq \frac{1}{2^i} + W_2^2( \meas, \meas_i).
\]
Since the measures $\meas_i$ are nonatomic,
their existence is guaranteed by \cite{Pratelli-Equality-2007}. 
Define the map $\pi_i:L^2(\meas) \to L^2(\meas_i)$ by
\[
\pi_i(v) := v \circ T_i.
\]
Note that $(T_i)_\#\meas_i=\meas$ grants that the map $\pi_i$ is a linear isometry for every $i$.

We claim that when $v_i \to v$ in $L^2(\meas)$, then $\pi_i(v_i) \to v$, strongly in $L^2$.  
According to the definition of $L^2$-strong convergence, we need to verify the following two statements:
\begin{enumerate}
\item \label{i:strongweakpart} For every $\xi \in \Cb(Z)$, 
\[
\int v_i \circ T_i \xi d \meas_i \to \int v \xi d \meas.
\]
\item \label{i:strongstrongpart} The following inequality holds 
\[
\limsup_{i \to \infty} \| \pi_i(v_i) \|_{L^2(\meas_i)} \leq \| v \|_{L^2(\meas)}.
\] 
\end{enumerate}
Since the map $\pi_i$ is an isometry, item (\ref{i:strongstrongpart}) follows immediately.
In order to check (\ref{i:strongweakpart}) we will argue by compactness, since any subsequence
of $v_i\circ T_i$ admits weakly $L^2$-convergent subsequences. By density, it suffices then
to check the property when $\xi$ is Lipschitz and with bounded support. Writing
$$
\int v_i \circ T_i \xi d \meas_i =
\int v_i \circ T_i \xi\circ T_i d \meas_i +
\int v_i \circ T_i (\xi-\xi\circ T_i) d \meas_i,
$$
we notice that the first term corresponds to $\int v_i\xi d\meas$, which obviously converges to
$\int v\xi d\meas$. On the other hand, using the fact that the Lipschitz constant of $\xi$ w.r.t.
$\tilde\rho$ can be estimated with $\max\{2\sup|\xi|,{\rm Lip}(\xi)\}$, the modulus of the second one can 
be estimated from above with
$$
\max\{2\sup|\xi|,{\rm Lip}(\xi)\}\|\tilde\rho(Id,T_i)\|_{L^2(\meas_i)}\|v_i\|_{L^2(\meas)}
$$
which, by our choice of $T_i$, converges to 0.
We have therefore shown that $\pi_i(v_i) \to v$ strongly in $L^2$.

Finally, note that the map $w \mapsto h_t^i w$ is $1$-homogeneous, 
see \cite[Rem.~4.14]{Ambrosio-Gigli-Savare11}, therefore $h_t^i\circ\pi_i$ is 
$1$-homogeneous as well.
\end{proof}

\begin{lemma}
	\label{le:ControlL2} 
Let $V \subset S(L^2(\meas))$ be compact. Then, for every $\epsilon > 0$ there exist $t> 0$ and an integer $i_0$ such that
\begin{equation}\label{eq:mozzano}
\| h_t^i \pi_i(v) - \pi_i(v) \|_{L^2(\meas_i)}   < \epsilon , \text{  for all $v \in V$, $i\geq i_0$.}
\end{equation}
\end{lemma}
\begin{proof}
The compactness of $V$, together with the continuity of $h_t$ at $t=0$ and its contractivity grant, for any $\epsilon > 0$,  
the existence of $t > 0$ satisfying
\[
\| h_{t} v - v\|_{L^2(\meas)}  < \epsilon, \text{ for all } v \in V.
\]
We now claim that with this choice of $t > 0$, \eqref{eq:mozzano} holds that for $i_0$ large enough.
Indeed, suppose not. By compactness of $V$, 
there exist a subsequence $(v_i) \subset V$ and $v \in V$ such that $v_i \to v$ in $L^2(\meas)$ and
\[
\| h_{t}^i \pi_i(v_i)  - \pi_i(v_i)\|_{L^2(\meas_i)} \geq \epsilon.
\]
However, $\pi_i(v_i) \to v$ strongly in $L^2$ and therefore also
$h_{t}^i(\pi_i(v_i)) \to h_{t} v$ strongly in $L^2$. 
These facts yield a contradiction.
\end{proof}

\noindent{\bf Proof of the upper semicontinuity in Theorem~\ref{tmain}.}
We can assume that $X_i$ does not consist of a single point for $i$ large enough
(otherwise, also the limit space consists of a single point and we need only to consider those
$i$ for which $X_i$ is not a single point in the argument below).
Let $\delta > 0$ and assume without loss of generality that $\lambda_k(\Ch)<\infty$. 
We will construct sets $V_i \in \mathcal{F}_k(L^2(\meas_i))$ such that
\[
\limsup_{i \to \infty} \sup_{u \in V_i} \Ch^i (u) \leq \lambda_k(\Ch) + 2 \delta,
\]
from which the lemma follows immediately.

By the definition of $\lambda_k(\Ch)$ there exists $V \in \mathcal{F}_k(L^2(\meas))$ such that 
\begin{equation}\label{eq:zan1}
\sup_{u \in V} \Ch(u) < \lambda_k(\Ch) + \delta.
\end{equation}
Let $\epsilon > 0$, so small that $(1 - \epsilon) ( \lambda_k(\Ch) + 2 \delta) >  \lambda_k(\Ch)+\delta$. 
Using Lemma~\ref{le:ControlL2}, we find $t > 0$ and an integer $i_0$ such that
\[
\| h_t^i \pi_i (v) \|_{L^2(\meas_i)} > 1 - \epsilon \quad \text{ for all $v \in V$, $i\geq i_0$.}
\]
Recall that, according to Lemma~\ref{le:ContructionIsometry}, the map $h_t^i \pi_i : L^2(\meas) \to L^2(\meas_i)$ is 
continuous and $1$-homogeneous (in particular, it is odd). 
Therefore 
$$
V_i := \left\{\frac{h_t^i \pi_i(v)}{\|h_t^i\pi_i(v)\|_2}:\ v\in V\right\} 
\in \mathcal{F}_k(L^2(\meas_i))
$$
for all $i\geq i_0$.

We are left to show that
\[
\limsup_{i \to \infty} \sup_{u \in V_i} \Ch^i(u) < \lambda_k(\Ch) + 2 \delta.
\]
We argue by contradiction: suppose that for a subsequence, that we will not denote differently, there exist functions $w_i \in V$ such that
\[
\Ch^i (\frac{h_t^i \pi_i(w_i)}{\|h_t^i\pi_i(w_i)\|_2}) \geq \lambda_k(\Ch) + 2 \delta
\]
and therefore
\[
\Ch^i( h_t^i \pi_i(w_i) ) \geq (1 - \epsilon) ( \lambda_k(\Ch) + 2 \delta) > \lambda_k(\Ch)+\delta.
\]
By compactness, we may also assume without loss of generality that $w_i \to w$ in $L^2(\meas)$
for some $w\in V$, so that Lemma~\ref{le:ContructionIsometry} gives
\[
\Ch^i(h_t^i\pi_i(w_i)) \to \Ch(h_t w).
\]
Since the map $h_t$ decreases the energy, we obtain
\[
\Ch( w ) \geq \Ch(h_t w) = \lim_{i \to \infty} \Ch^i(h_t^i\pi_i(w_i)) 
\geq \lambda_k(\Ch)+\delta,
\]
which is a contradiction with \eqref{eq:zan1}.

\section{Lower semicontinuity of the spectrum}

As in the proof of Lemma~\ref{le:ContructionIsometry}, since $\meas$ is nonatomic we can find 
$2^{-i}$-almost optimal transport maps $S_i: X \to X_i$ from $\meas$ to $\meas_i$ (so that 
$(S_i)_\#\meas=\meas_i$ and $\int\tilde\rho^2(Id,S_i)d\meas\to 0$, with
$\tilde\rho=\min\{1,\rho\}$) 
to define isometries $\sigma_i: L^2(\meas_i) \to L^2(\meas)$ by 
\begin{equation}
\label{eq:DefIsom-L2mi-to-L2m}
\sigma_i (w) = w \circ S_i.
\end{equation}

\begin{lemma}
\label{le:TransferL2WeakConvergence}
Let $f_i \in L^2(\meas_i)$ with $\sup_i\|f_i\|_{L^2(\meas_i)}<\infty$ and $f \in L^2(\meas)$.
Then $f_i \to f$ weakly in $L^2$ if and only if $\sigma_i(f_i) \to f$ weakly in $L^2(\meas)$.
\end{lemma}

\begin{proof}
Thanks to the compactness properties of weak $L^2$ convergence, we need only
to test the convergence in duality with Lipschitz functions with bounded support $\xi:Z\to\mathbb{R}$.

We will show below the convergence
\begin{equation}
\label{eq:ReplaceXi}
\int f_i\circ S_i \xi d \meas - \int f_i \circ S_i \xi \circ S_i d \meas \to 0
\end{equation}
as $i\to \infty$. Given this fact, 
\[
\int f_i\circ S_i  \xi d \meas \to \int f \xi d \meas
\] 
as $i \to \infty$ if and only if
\[
\int f_i\xi d\meas_i= \int f_i\circ S_i \xi\circ S_i d \meas \to \int f \xi d \meas
\]
as $i\to \infty$, from which the Lemma follows.
 
The convergence (\ref{eq:ReplaceXi}) follows at once by estimating, as we did in the proof
of Lemma~\ref{le:ContructionIsometry}, the difference with
$$
\max\{2\sup|\xi|,{\rm Lip}(\xi)\}\|\tilde\rho(Id,S_i)\|_{L^2(\meas)}\|f_i\|_{L^2(\meas_i)}
$$ 
and using the uniform boundedness of the $L^2$ norms of $f_i$.
\end{proof}

\begin{corollary}
	\label{co:TransferL2StrongConvergence}
	Let $f_i \in L^2(\meas_i)$ and $f \in L^2(\meas)$.
	Then $f_i \to f$ strongly in $L^2$ if and only if $\sigma_i(f_i) \to f$ strongly in $L^2(\meas)$.
\end{corollary}
\begin{proof}
This follows from the previous Lemma~\ref{le:TransferL2WeakConvergence}, and from
the fact that the maps $\sigma_i: L^2( \meas_i) \to L^2(\meas)$ are isometries.
\end{proof}

\begin{lemma}
	\label{le:UniformTotalBoundedness}
	Assume that \eqref{eq:wlsi} holds with $A,\,B\geq 0$, $\bar x\in Z$ independent of $i$. 
	Then, for all $s,\,t\geq 0$ the sublevel sets 
\begin{equation}\label{eq:mozzano1}
	E_i^{s,t}:=\left\{u \in L^2(\meas_i) : \ \|u\|_{L^2(\meas_i)} \le s,\,\,\Ch^i(u) \le t \right\} 
\end{equation}
 are uniformly totally bounded in $L^2(\meas_i)$. 
	That is, for every $\epsilon > 0$ there exists an integer $N$ such that for every 
	$i \in \mathbb{N}$ one can find points $p^i_1, \ldots, p^i_N \in E_i^{s,t}$ satisfying
	\[
	E_i^{s,t} \subset \bigcup_{j=1}^N B_\epsilon(p^i_j).
	\]
\end{lemma}

\begin{proof}
	Suppose not, then there exist $\epsilon > 0$, a subsequence (that we do not relabel) and points 
	$p_j^i \in E_i^{s,t}$, $j = 1, \ldots, N(i)$, with $N(i)\to\infty$ as $i\to\infty$, such that 
	\[
	\| p_j^i - p_k^i  \|_{L^2(\meas_i)} > \epsilon\qquad\forall 1\leq j<k\leq N(i).
	\]
	Since $j\leq N(i)\to\infty$, by a diagonal construction and the 
	compactness result \cite[Thm.~6.3]{Gigli-Convergence-2015} stated in Theorem~\ref{tcompa}, we may assume that for all $j$ fixed the functions
	$p_j^i\in L^2(\meas_i)$ $L^2$ strongly converge to $p_j \in L^2(\meas)$ as $i\to\infty$, and thus by Corollary~\ref{co:TransferL2StrongConvergence}
	we obtain that $\sigma_i(p_j^i) \to p_j$
	strongly in $L^2(\meas)$. Moreover the $\liminf$-inequality in the Mosco convergence yields $p_j \in E^{s,t}$,
	with $E^{s,t}$ as in \eqref{eq:mozzano6}.
	
Then we have $\|p_j-p_k\|_{L^2(\meas)}\ge \epsilon$ for all $j \neq k$, which contradicts the compactness of $E^{s,t}$ with respect to the $L^2(\meas)$-norm.
\end{proof}

\begin{lemma}
	\label{le:HausdorffCompactness}
Assume that \eqref{eq:wlsi} holds with $A,\,B\geq 0$, $\bar x\in Z$ independent of $i$. 
Let $V_i \subset E_i^{s,t}$, with $E_i^{s,t}$ as in \eqref{eq:mozzano1}. 
Then there exist a subsequence $i(j)$ and a compact subset $V \subset L^2(\meas)$ such that 
\[
\sigma_{i(j)} (V_{i(j)}) \to V\quad\text{in $L^2(\meas)$}
\]
in the Hausdorff distance. 
\end{lemma}
\begin{proof}
As we have established the uniform total boundedness of the sublevel sets $E_i^{s,t}$ in 
Lemma~\ref{le:UniformTotalBoundedness}, the proof now follows from a standard construction:
For $\epsilon_k = 1/k$ we select a sequence $\epsilon_k$-nets of $V_i$ as provided by 
Lemma~\ref{le:UniformTotalBoundedness}, that is points $p_{j,k}^i \in V_i$, $j = 1, \ldots, N_k$, such that
\[
V_i \subset \bigcup_{j=1}^{N_k} B_{\epsilon_k} (p_{j,k}^i).
\]
As in the proof of Lemma~\ref{le:UniformTotalBoundedness}, for fixed $k$, the sequences of images under $\sigma_i$ of these $\epsilon_k$-nets is strongly compact in $L^2(\meas)$. Using the limit points $p_{j,k}$
of $p^i_{j,k}$ along a suitable subsequence of indices $i$, one can define
$$
V:=\bigcap_{k=1}^\infty\bigcup_{j=1}^{N_k}\overline{B}_{1/k}(p_{j,k}).
$$
\end{proof}

\noindent{\bf Proof of the lower semicontinuity in Theorem~\ref{tmain}.} Let 
$V_i$ be compact sets in $L^2(\meas_i)$ such that $V_i \subset S(L^2(\meas_i))$ and 
\[
\sup_{u \in V_i}  \Ch^i(u) \leq \lambda_k(\Ch_i) + 1/i.
\]
Let $\sigma_i: L^2(\meas_i) \to L^2(\meas)$ be the linear isometries defined by \eqref{eq:DefIsom-L2mi-to-L2m}.
By Lemma~\ref{le:HausdorffCompactness}, there exist a compact subset $V \subset L^2(\meas)$ and a subsequence 
such that $\sigma_i(V_i) \to V$ in the Hausdorff distance. 
Since the sets $V_i$ are symmetric, the set $V$ is symmetric as well.
Therefore, by Proposition~\ref{pr:SemicontinuityGenus}, $\gamma(V) \geq \limsup_{i \to \infty} \gamma(V_i) \geq k$. 
Hence $V \in \mathcal{F}_k(L^2(\meas))$.

Moreover, it follows from Corollary~\ref{co:TransferL2StrongConvergence} that for every $v \in V$, 
there exists a sequence $v_i \in V_i$  strongly converging to $v$ in $L^2$.

By the $\liminf$-inequality in the Mosco convergence,
\[
\liminf_{i \to \infty} \Ch^i(v_i) \geq \Ch(v).
\]
Hence,
\[
\liminf_{i \to \infty} \sup_{ u \in V_i} \Ch^i(u) \geq \sup_{v \in V} \Ch(v),
\]
which, taking our choice of $V_i$ into account, immediately gives \eqref{eq:mozzano4}.

\section{Existence of eigenfunctions and eigenvalues}

\subsection{Gradient flow of Cheeger energy on sphere}

In this section we note the well-posedness of the gradient flow of Cheeger's energy \emph{on the sphere} $S(L^2(\meas))$,
the latter denoted by $S$ for simplicity of notation. In this section, we shall also denote by $\overline{B}$ the closed unit ball of
$L^2(\meas)$ and denote by $\|\cdot\|$ the norm of $L^2(\meas)$, by $\langle\cdot,\cdot\rangle$ the scalar product.

For a constant $M > 0$ we denoted by $E^{1,M}$ the sublevel set of the energy intersected with $\overline{B}$, namely
\begin{equation}\label{eq:SublevelSetCheegerEnergy2}
E^{1,M} := \{ u \in \overline{B}:\ \Ch(u) \leq M \}.
\end{equation}
We consider the functional $\Phi:L^2(\meas)\to [0,\infty]$ defined by
$$
\Phi(u):=
\begin{cases}
\Ch(u) &\text{if $u\in E^{1,M}$;}
\cr\cr
+\infty &\text{otherwise.}
\end{cases}
$$
Since we are going to study $\Phi$ only with a specific choice of the constant $M$ (specifically any
$M=\lambda+1$, with $\lambda$ the energy level we are interested in) we will not emphasize the
dependence of $\Phi$ on this constant. 

Since $E^{1,M}$ is closed and convex, it is easily seen that $\Phi_L:=\Phi-L\|\cdot\|^2$ is $(-2L)$-convex
and lower semicontinuous. In addition, the closure of the finiteness domain of $\Phi_L$ is $E^{1,M}$.
Therefore, the theory of gradient flows for semiconvex and lower semicontinuous
functionals applies, and provides a unique continuous gradient semigroup ${\sf S}_t$ of
$\frac 12\Phi_L$ on $E^{1,M}$, namely a locally absolutely continuous map $u(t)=:{\sf S}_tu$ in $(0,\infty)$
satisfying $u(t)\to u=:{\sf S}_0u$ as $t\to 0$ and 
$$
u'(t)\in -\partial_F\frac 12\Phi_L(u(t)) \qquad\text{for $\Leb^1$-a.e. $t>0$,}
$$
where  $\partial_F \frac 12\Phi_L(u)$ is the Frechet subdifferential of $\frac 12\Phi_L$ at $u$, defined in \eqref{def:frediff}.

We recall that semigroup ${\sf S}_t$ satisfies the contractivity property \eqref{eq:contractivity}, as well as
the regularizing properties \eqref{eq:regularizing1}, \eqref{eq:regularizing2}.

\begin{lemma} \label{lem:differentials} For all $u\in\overline{B}\setminus S$ with $\Ch(u)<M$ one has
	\begin{equation}\label{eq:first_subdifferential}
	\partial_F\Phi_L(u)=\partial\Ch(u)-2Lu\ni -2(\Delta u+Lu)
	\end{equation}
    while for all $u \in S$ with $\Ch(u)<M$ it holds that
	\begin{equation}
	\partial_F \Phi_L(u) = \partial \Phi(u) - 2 L u = \partial \Ch(u) - 2 L u + \{ 2 \mu u : \ \mu \geq 0 \}.
	\end{equation}
	As a consequence, for all $u \in S$ with $\Ch(u) < M$, the element of minimal norm in $\partial_F \Phi_L(u)$ is given by $- \Delta u - \Ch(u) u$ and $-\Delta u=\Ch(u)u$ whenever $|\partial\Phi_L|(u)=0$.
\end{lemma}
\begin{proof} To simplify notation, we prove \eqref{eq:first_subdifferential} in the case when $L=0$, so that
	$\Phi_0=\Phi$ (with no loss of generality, since $L\|\cdot\|^2$ is a smooth perturbation); in this case, the convexity of $\Phi$ grants
	the identity $\partial_F\Phi=\partial\Phi$, so we need only to prove that $\partial\Phi(u)=\partial\Ch(u)$. Since $\Phi\geq\Ch$,
	the inclusion $\supset$ is obvious, since $\Phi(u)=\Ch(u)$ by assumption. Conversely, if $\xi\in\partial\Phi(u)$ and
	$v\in D(\Ch)$, one has
	$$
	\Phi(u+tv)\geq \Phi(u)+t\langle \xi ,v \rangle=\Ch(u)+t\langle\xi, v\rangle.
	$$
	Since $\|u\|<1$ and $\Ch(u)<M$, for $t>0$ sufficiently small one has $u+tv\in E^{1,M}$, and then
	$$\Ch(u+tv)\geq \Ch(u)+t\int\xi v d\meas.$$
	By monotonicity of difference quotients, since $v\in D(\Ch)$ is arbitrary, one then obtains that $\xi\in\partial\Ch(u)$.
	
	To prove the second statement, we first claim the following: If $u \in S$ with $\Ch(u)< M$ and $\zeta \in \partial \Phi(u)$, then also
	\[
	\zeta - \langle\zeta,u\rangle u + 2 u \Ch(u) \in \partial \Phi(u) \cap \partial \Ch(u).
	\]
	This is a consequence of the $2$-homogeneity of the Cheeger energy. 
	Indeed, let $w \in L^2(\meas) \cap D(\Ch)$. Then 
	\[
	\begin{split}
	\Ch(u + t w) 
	&= \| u + t w\|^2 \Ch( \frac{u + tw }{\|u + tw\|} ) \\
	&\geq (1 + 2t \langle u,w\rangle )\left(\Ch(u) + \int \zeta \left( \frac{u + t w}{\|u + t w\|} - u \right)d \meas \right) + o (t) \\
	&= \Ch(u) + 2t\langle u,w\rangle\Ch(u) + t \int \zeta ( w- \langle w,u\rangle u)d \meas + o(t) \\
	& = \Ch(u) + 2 t\langle u,w\rangle\Ch(u) + t \int (\zeta - \langle\zeta, u\rangle u )w d \meas + o(t),
	\end{split}
	\]
	where in the second line we used the assumption that $\zeta \in \partial \Phi(u)$. 
	It follows that 
	\[
	\xi := \zeta - \langle\zeta,u\rangle u + 2 u \Ch(u) \in \partial \Phi(u) \cap \partial \Ch(u),
	\]
	proving the claim. From the $2$-homogeneity of $\Ch$, considering variations
	$u_\epsilon=(1-\epsilon)u$ with $\epsilon\to 0^+$, it also follows that 
	\[
	\langle \zeta ,u\rangle \geq 2 \Ch(u)
	\]
	so that 
	\[
	\zeta = \xi + 2 \mu u\quad\text{with}\quad\mu:=\frac 12\langle\zeta,u\rangle-\Ch(u)\geq 0.
	\]
	Therefore,
	\[
	\partial_F \Phi_L(u) = \{ \xi + 2(\mu - L) u : \ \xi \in \partial \Ch(u), \mu \geq 0 \}.
	\]
	
	Recall that $-2 \Delta u$ is the element of minimal norm in $\partial \Ch(u)$ and that for every $\xi \in \partial \Ch(u)$ it holds that
	\[
	\int \xi u d \meas = 2 \Ch(u).
	\]
	As a consequence,
	\[
	\partial_F \Phi_L(u) = \{ \xi - 2 \Ch(u) u + 2(\mu + \Ch(u) - L) u : \ \xi \in \partial \Ch(u),\,\, \mu \geq 0 \}
	\]
	where 
	\[
	\xi - 2 \Ch(u) u \perp u.
	\]
	We conclude that the element of minimal norm in $\partial \Phi_L(u)$ is given by
	\[
	- 2 \Delta u - 2 u \Ch(u). 
	\]
	Therefore, if the descending slope $|\partial \Phi_L|(u)$ vanishes, one has
	$ -\Delta u = u \Ch(u)$.
\end{proof}

\begin{theorem}\label{thm:goodsphere} If $L>M$ the semigroup ${\sf S}_t$ leaves $S\cap E^{1,M}$ invariant, namely
${\sf S}_tu\in S\cap E^{1,M}$ whenever $u\in S\cap E^{1,M}$. In addition, for all $t>0$, ${\sf S}_t$ maps $S\cap E^{1,M}$ 
continuously to $H^{1,2}(\meas)$:  more precisely, if $f_i\in S \to f$ in $L^2(\meas)$, then 
\[
\|{\sf S}_t f_i- {\sf S}_tf\|\to 0\qquad\text{and}\qquad\Ch( {\sf S}_t f_i ) \to \Ch( {\sf S}_tf ).
\]
\end{theorem}
\begin{proof} Let $u\in S\cap E^{1,M}$ and set $u(t)={\sf S}_tu$; notice that $\Phi_L(u(t))<\infty$ for all $t>0$
implies $\Ch(u(t))\leq M$ for all $t\geq 0$; for $\Leb^1$-a.e. $t>0$ such that $\Ch(u(t))<M$ and 
$1\geq\|u(t)\|^2\geq M/L$ one has (using \eqref{eq:nicelaplacian} and Lemma~\ref{lem:differentials})
$$
\frac{d}{dt}\frac 12 \|u(t)\|^2=\int u(t)(\Delta u(t)+Lu(t)) d\meas\geq -M+L\|u(t)\|^2\geq 0.
$$
On the other hand, for $\Leb^1$-a.e. $t>0$ such that $\Ch(u(t))=M$ the monotonicity of $\Phi_L(u(t))$
still implies that $\frac{d}{dt} \|u(t)\|^2\geq 0$. Hence, $\frac{d}{dt} \|u(t)\|^2\geq 0$ for $\Leb^1$-a.e. $t>0$ such that 
$1\geq\|u(t)\|^2\geq M/L$. This immediately implies that $\|u(t)\|=1$ for all $t\geq 0$.

The continuity statement follows immediately from the contractivity of ${\sf S}$, the uniform upper bounds on
the slope \eqref{eq:regularizing2} and the representation formula \eqref{eq:RepresentationLocalSlope}.
(See also the proof of Lemma~\ref{le:LikePalaisSmale} below.)
\end{proof}

Motivated by the previous result, 
in the sequel we denote by $h^S_t$ the restriction of the semigroup ${\sf S}_t$ to $S\cap E^{1,M}$. 
 
\subsection{A Palais-Smale condition}

Under the compactness assumption on the sublevel sets $E^{1,M}$, the following compactness holds for
sublevel sets for the energy and slope of $\Phi$ (or, equivalently, of any of the $\Phi_L$, because of the
inequalities $|\partial\Phi|-2L\leq |\partial\Phi_L|\leq |\partial\Phi|+2L$), 
which plays the role of the celebrated Palais-Smale condition in our setting.

\begin{lemma}
\label{le:LikePalaisSmale}
Assume that the set $E^{1,M}$ in \eqref{eq:SublevelSetCheegerEnergy2} is compact in $L^2(\meas)$. Then, 
for every $T \geq 0$, the subset
\[
\mathcal{E}^{M,T} = \{ u \in S : \  \Phi(u) \leq M,\,\,  | \partial\Phi|(u) \leq T \}
\]
is compact in $L^2(\meas)$ and in $H^{1,2}(\meas)$.
More precisely, for every sequence of functions $(v_j) \subset \mathcal{E}^{M,T}$ there exist a subsequence 
(that we do not denote differently) and a function $v \in \mathcal{E}^{M,T}$ such that $v_j \to v$ in $L^2(\meas)$ and 
\[
\Ch(v_j) \to \Ch(v).
\]
\end{lemma}

\begin{proof}
Let $(v_j) \subset \mathcal{E}^{M,T}$ be a sequence of functions. 
By the assumed compactness of $E^{1,M}$, this sequence has an $L^2(\meas)$-strongly converging subsequence (that we will not denote differently) $v_j \to v$, for some $v \in S$. By the lower semicontinuity of $\Ch$ and of the local slope, 
it follows that $v \in \mathcal{E}^{M,T}$.
 
On the other hand, it follows by the representation \eqref{eq:RepresentationLocalSlope} of the local slope that
\[
\Phi(v_j) - \Phi(v) \leq T \|v_j-v\|.
\]
Consequently, recalling also the definition of $\Phi$, one has also
\[
\limsup_{j \to \infty} \Ch(v_j) \leq \Ch(v).
\]
By the lower semicontinuity of the Cheeger energy, we conclude.
\end{proof}

\subsection{Existence of eigenvalues}

For $\lambda \geq 0$, set $M=\lambda+1$, fix $L>M$ and 
denote by $K_\lambda$ the set of critical points at energy level $\lambda$ of $\Phi_L=\Phi-L\|\cdot\|^2$, namely
\[
K_\lambda := \{ u \in S :\ \ \Phi(u) = \lambda, \,\, |\partial\Phi_L|(u) = 0\}.
\]
The set $K_\lambda$ is compact by Lemma~\ref{le:LikePalaisSmale} and recall that Lemma~\ref{lem:differentials}
shows that $-\Delta u=u\Ch(u)$ for all $u\in K_\lambda$.

For $r> 0$ we define the tubular neighborhoods of $K_\lambda$
\begin{equation}
\label{eq:DefinitionU-nbh}
U_{\lambda,r} := \{ v \in S:\ \|u-v\| < r \text{ for some } u \in K_\lambda \}.
\end{equation}
Note that by compactness of $K_\lambda$, for every neighborhood $N$ of $K_\lambda$, there 
exists $r>0$ such that $U_{\lambda, r} \subset N$.

Moreover, define the sets (not necessarily neighborhoods) 
\begin{equation}
\label{eq:DefinitionN-set}
N_{\lambda,\delta} := \{ u \in S:\ \ |\Phi(u) - \lambda| \leq \delta, \,\, |\partial\Phi_L|^2(u) \leq 4\delta \}.
\end{equation}
Note that, thanks to Lemma~\ref{le:LikePalaisSmale} and to the lower semicontinuity of the slope, 
for every neighborhood $U$ of $K_\lambda$, there exists 
$\delta > 0$ such that $N_{\lambda, \delta} \subset U$, and we will apply this property to the sets $U_{\lambda,r}$.

We recall that
\[
\lambda_k(\Ch) = 
\inf_{V \in \mathcal{F}_k(L^2(\meas)) } \sup_{u \in V}\Ch (u). 
\]

The next theorem states that $\lambda_k(\Ch)$ is an eigenvalue for every $k \geq 1$, that is there exists
$u \in S$ such that
\begin{equation}\label{eq:calloo}
- \Delta u = \lambda_k(\Ch) u.
\end{equation}
Additionally, it gives a statement about the multiplicity of such eigenvalues. In the proof we will use the restriction $h_t^S$ 
to the sphere $S$ of the gradient flow ${\sf S}_t$ of $\Phi_L$.  

\begin{theorem}
\label{th:EigenvalueExistence}
Assume that the sublevel sets $E^{1,M}$ of the Cheeger energy, as defined in 
\eqref{eq:SublevelSetCheegerEnergy2}, are compact. If for some $k,\,\ell\geq 1$ it holds that
\[
\lambda_k(\Ch) = \lambda_{k+1}(\Ch) = \cdots = \lambda_{k + \ell -1}(\Ch) = \lambda
\]
then $\gamma(K_\lambda) \geq\ell$. In particular, $K_\lambda\neq\emptyset$ and therefore 
there exists $u \in S$ such that \eqref{eq:calloo} holds.
Finally, if $\ell>1$ there are infinitely many solutions $u$ to this equation.
\end{theorem}

\begin{proof}
Our proof is similar to the proof of Lemma 5.6 in \cite{Struwe-Variational-2008}, with some technical differences. 
While Struwe uses a pseudo-gradient flow, we may directly apply the gradient flow $h_t^S$.

By \eqref{eq:regularizing2} with $\lambda=-2L$ 
we obtain the existence of a constant $J>1$ such that for all 
$v \in S\cap E^{1,M}$ and all $t\in [1/2,2]$, one has
\begin{equation}
\label{eq:DefineJ}
|\partial\Phi_L|(h_t^S v) \leq J.
\end{equation}

There exists a symmetric, open neighborhood $N \supset K_\lambda$ such that 
$\gamma(\overline{N}) = \gamma(K_\lambda)$.
Moreover, there exist $r\in (0,1)$ and $\delta\in (0,1)$ such that
\[
K_\lambda\subset N_{\lambda,\delta J/r} \subset U_{\lambda, r} \subset U_{\lambda, 2 r} \subset N,
\]
where $U_{\lambda, r}$, and $N_{\lambda,\delta J/r}$ were defined in (\ref{eq:DefinitionU-nbh}) and (\ref{eq:DefinitionN-set}) respectively (this can be achieved first choosing $r$, and then choosing $\delta$).

Select a set $V_0 \subset \mathcal{F}_{k+\ell-1}(L^2(\meas))$, such that 
\[
\sup_{u \in V_0} \Ch(v) < \lambda + \delta.
\]	
Without loss of generality we may assume that for every $t \in [0,1]$, $h_t^S(V_0) \subset \mathcal{E}^{\lambda + \delta, J}$, otherwise we can just replace $V_0$ by $h_{1/2}^S(V_0)$, and use (\ref{eq:DefineJ}) (notice that continuity and 1-homogeneity
of $h_{1/2}^S$ ensure that still $h_{1/2}^S(V_0)\in\mathcal{F}_{k+\ell-1}(L^2(\meas))$).

We now claim that
\begin{equation}\label{eq:calloo1}
h_1^S(V_0) \subset E^{1,\lambda - \delta} \cup N.
\end{equation}

Indeed, assume by contradiction that $h_1^Sv\notin N$ for some $v\in V_0$, with $\Ch(h_1^S v)>\lambda-\delta$.
Then, by monotonicity, one has $\Ch(h_t^Sv)>\lambda-\delta$ for all $t\in [0,1]$ and then $h_t^Sv\in N_{\lambda,\delta J/r}$
if and only if $|\partial\Phi_L|^2(h_t^Sv)|\leq 4\delta J/r$. Now, if for all $t \in [0,1]$ one has $h_t^S v \notin N_{\lambda,\delta J/r}$, 
energy dissipation gives
\[
\Ch(v) - \Ch(h_1^Sv) \geq 4\delta J / r >  4\delta.
\] 
If there exists $t \in [0,1]$ such that $h_t^Sv \in N_{\lambda,\delta J/r}$, then since $h_1^Sv\notin N$ we must have
(since the velocity of the curve is at most $J$) 
$$
\Leb^1(\{ s \in (t,1) :\ h_s^Sv \in U_{\lambda,2r} \setminus U_{\lambda, r} \})\geq r/J.
$$
Hence, using once more the fact that $ h_s^Sv \notin U_{\lambda, r}$ implies
$|\partial\Phi_L|^2(h_s^Sv)>4\delta J/r$, we get
\[
\begin{split}
\Ch(v) - \Ch(h_1^Sv) 
&\geq 4\delta (J/r) \mathcal{L}^1(\{ s \in (t,1) : \ h_s^Sv \in U_{\lambda,2r} \setminus U_{\lambda, r} \}) \\
&\geq 4\delta (J/r) (J/r)^{-1} = 4\delta.
\end{split}
\]
This proves the claim.

By the very definition of $\lambda_k(\Ch)$, one has 
\[
\gamma(\overline{E^{\lambda - \delta} }) < k.
\]
Hence, combining this inequality with \eqref{eq:calloo1} and
using the subadditivity of the genus stated in Proposition~\ref{pr:PropertiesGenus}, we find
\[
\begin{split}
\gamma(\overline{N}) &\geq \gamma(\overline{E^{\lambda - \delta} \cup N})
- \gamma(\overline{E^{\lambda - \delta}})\\
&> \gamma(h_1^S(V_0)) - k \geq \gamma(V_0) - k \\
&\geq k + \ell - 1 - k = \ell- 1
\end{split}
\]
so that 
\[
\gamma(K_\lambda) = \gamma(\overline{N}) \geq\ell.
\]
It follows that $K_\lambda \neq \emptyset$. 
Finally, $K_\lambda$ is a finite set if and only if $\gamma(K_\lambda) \leq 1$. 
\end{proof}

%
%

\end{document}